\documentclass[runningheads]{llncs}
\usepackage{graphicx}

\usepackage{amsmath,amssymb,booktabs,comment,longtable}
\hypersetup{
colorlinks=true}

\newtheorem{notation}[theorem]{Notation}

\newcommand{\id}{\operatorname{id}}
\newcommand{\ev}{\mathop{\mathrm{ev}}\nolimits}

\newcommand{\R}{\mathbb{R}}

\newcommand{\C}{\mathbb{C}}

\newcommand{\pairing}[2]{\langle #1, #2 \rangle}
\newcommand{\Sym}{\operatorname{Sym}}
\newcommand{\trans}{{}^t\!}

\newcommand{\spanvec}{\operatorname{span}}

\newcounter{mycounter}

\makeatletter
    
    \@addtoreset{equation}{section}
\makeatother
\begin{document}
\title{On a method to construct exponential families by representation theory}
%\titlerunning{Abbreviated paper title}
% If the paper title is too long for the running head, you can set
% an abbreviated paper title here
%
\author{Koichi Tojo\inst{1} \and
Taro Yoshino\inst{2}}
\authorrunning{K.~Tojo \and T.~Yoshino}
% First names are abbreviated in the running head.
% If there are more than two authors, 'et al.' is used.
%
\institute{RIKEN Center for Advanced Intelligence Project, Tokyo, Japan/ \\ Department of Mathematics, Faculty of Science and Technology, Keio University, 3-14-1 Hiyoshi, Kohoku-ku, Yokohama, 223-8522, Japan\\
\email{koichi.tojo@riken.jp}\\
 \and
Graduate School of Mathematical Science, The University of Tokyo,\\ 3-8-1 Komaba, Meguro-ku, Tokyo 153-8914, Japan\\
\email{yoshino@ms.u-tokyo.ac.jp}}
\maketitle              % typeset the header of the contribution
\begin{abstract}
%The abstract should briefly summarize the contents of the paper in
%150--250 words.
Exponential family plays an important role in information geometry. 
In \cite{TY18}, we introduced a method to construct an exponential family $\mathcal{P}=\{p_\theta\}_{\theta\in\Theta}$ on a homogeneous space $G/H$ from a pair $(V,v_0)$. 
Here $V$ is a representation of $G$ and $v_0$ is an $H$-fixed vector in $V$. 
Then the following questions naturally arise: 
(Q1) when is the correspondence $\theta\mapsto p_\theta$ injective?
(Q2) when do distinct pairs $(V,v_0)$ and $(V',v_0')$ generate the same family?
In this paper, we answer these two questions (Theorems~\ref{theorem:main} and \ref{theorem:equiv}). 
Moreover, in Section~\ref{sec:example}, we consider the case $(G,H)=(\R_{>0}, \{1\})$ with a certain representation on $\R^2$. 
Then we see the family obtained by our method is essentially generalized inverse Gaussian distribution (GIG). 


\end{abstract}
\section{Introduction}
Let $G$ be a Lie group and $H$ its closed subgroup. 
In \cite{TY18}, we introduced a method to construct an exponential family 
$\mathcal{P}=\{p_\theta\}_{\theta\in \Theta}$ on the homogeneous space $X:=G/H$ from $(V,v_0)$. 
In this paper, we answer two natural questions on our method. 

\subsection{Correspondence parameters and probability measures}
In the theory of exponential family, 
``minimal representation'' is important (\cite{Barndorff-Nielsen70}). 
If an exponential family is realized by ``minimal representation'', 
then we obtain one-to-one correspondence between the parameter space and the family of probability measures, which enable us to make use of the family. 
Moreover, from the perspective of information geometry, the correspondence is used as a coordinate. 
Then we would like to consider the following:
\begin{question}\label{question:injective}
When is the following correspondence injective?
\begin{align} 
\Theta \ni \theta \mapsto p_\theta\in \mathcal{P}. \label{eq:correspondence}
\end{align}
\end{question}
We want to answer this question for families obtained by our method. 
We give a necessary and sufficient condition for the injectivity of (\ref{eq:correspondence}) in Theorem~\ref{theorem:main}. 
It is, however, a little bit difficult to check. 
So, we will see the following easier equivalent conditions (A) and (B) are necessary. 
\begin{enumerate}
\item[(A)] The orbit $Gv_0$ is not contained in any proper affine subspace of $V$. 
\item[(B)]
\begin{enumerate}
\item[(1)] $v_0$ is cyclic, 
\item[(2)] $V^\vee$ has no nonzero $G$-fixed vector. 
\end{enumerate}
\end{enumerate}
In the case where $G$ is compact or connected semisimple, 
they are also sufficient (see Remark~\ref{rem:omega_trivial}). 
\subsection{Equivalence relation}
Our method in \cite{TY18} constructs an exponential family from a pair $(V,v_0)$. 
In some cases, the same exponential family comes from distinct pairs $(V,v_0)$ and $(V',v_0')$. 
To reduce the choice of $(V,v_0)$, it is useful to give an answer to the following question. 
\begin{question}\label{question:equiv_rel}
When do distinct pairs $(V,v_0)$ and $(V',v_0')$ generate the same family?
\end{question}
We give an answer to this question in Theorem~\ref{theorem:equiv}. 
More precisely, we introduce an equivalence relation on the set of pairs $\{(V,v_0)\}$ and 
show that two families obtained by $(V,v_0)$, $(V',v_0')$ coincide if $(V,v_0)\sim (V',v_0')$. 

\section{Main theorems}
\subsection{Method introduced in \cite{TY18}}\label{sec:setting}
Before stating our main results, 
we recall the method introduced in \cite{TY18}. 
Let $G$ be a Lie group and $H$ its closed subgroup. 
Then the quotient space $X:=G/H$ naturally equips manifold structure, 
which is called the {\it homogeneous space} of $G$. 

Let $V$ be a finite dimensional real vector space, 
and $\rho\colon G\to GL(V)$ a Lie group homomorphism. 
Then the pair $V:=(\rho,V)$ is called a {\it representation} of $G$. 
We often use simpler notation $gv:=\rho(g)v$ for $g\in G$ and $v\in V$. 

A vector $v_0\in V$ is said to be $H$-{\it fixed} if $hv_0=v_0$ for any $h\in H$. 
We denote by $V^H$ the linear subspace consisting of all $H$-fixed vectors. 
Let $(V, v_0)$ be a pair of representation of $G$ and an $H$-fixed vector. 

We put 
\begin{align}
\Omega_0(G,H)&:=\{ \chi:G\to \R_{>0}\ |\ \chi \text{ is a continuous group homomorphism}, \chi|_H=1\},  \label{eq:omega_GH}\\
\log\Omega_0(G,H)&:=\{\log \chi:G\to \R\ |\ \chi\in \Omega_0(G,H)\}. \label{eq:log_omega}
\end{align}
Take a relatively $G$-invariant measure $\mu$ on $X$. 
Then we define a measure $\tilde{p}_\theta$ on $X$ parameterized by $V^\vee \times \Omega_0(G,H)$ as follows:
\begin{align}
d\tilde{p}_\theta(x)=d\tilde{p}_{\xi,\chi}(x)
:=\exp(-\pairing{\xi}{xv_0})\chi(x)d\mu(x)\quad (x\in X)\label{eq:tilde_p}, 
\end{align}
where $\theta=(\xi,\chi) \in V^\vee\times\Omega_0(G,H)$. 
\begin{remark}\label{remark:well_def_notion}
Since $v_0$ is $H$-fixed, 
the notion $xv_0$ in (\ref{eq:tilde_p}) is well-defined. 
Owing to $\chi|_H=1$, the notion $\chi(x)$ is also well-defined for $\chi\in \Omega_0(G,H)$. 
\end{remark}
Then we consider the normalization of the measures above. 
Put 
\begin{align}
\Theta&:=\{\theta= (\xi,\chi)\in V^\vee\times \Omega_0(G,H)\ |\ \int_X d\tilde{p}_\theta<\infty\},\label{eq:parameter_space}\\
\varphi(\theta)&:=\log \int_X d\tilde{p}_{\theta}\quad (\theta\in \Theta),\label{eq:normalizing_constant}\\
dp_\theta&:=e^{-\varphi(\theta)}d\tilde{p}_\theta. \label{eq:normalized_measure}
\end{align}
Then we obtain a family of distributions on $X$ as follows: 
\begin{align}
\mathcal{P}:=\{p_\theta \}_{\theta\in \Theta}. 
\end{align}
This is an exponential family if $\Theta\neq \emptyset$ (\cite{TY18}). 

\subsection{Correspondence}
In this section, we give an answer to Question~\ref{question:injective}. 
Namely, we state a criterion of the injectivity of the correspondence (\ref{eq:correspondence}). 
Moreover, we also give necessary conditions, which one can easily check (Proposition~\ref{prop:proper_affine_subspace}) %and Lemma~\ref{lem:nec_cond_rep}). 

\begin{theorem}\label{theorem:main}
In the setting as in Section~\ref{sec:setting}, 
the following three conditions are equivalent:
\begin{enumerate}
\item[$(i)$] The correspondence $\Theta\ni \theta \mapsto p_\theta \in \mathcal{P}$ is injective. 
%\item $\spanvec\{gg'v_0-gv_0-g'v_0+v_0\ |\ g, g'\in G\}=V$. 
\item[$(ii)$] There does not exist $\xi\in V^\vee\setminus\{0\}$ such that $f_\xi\in \log\Omega_0(G,H)$. 
\item[$(iii)$] There does not exist a triple $(\xi, \chi, c)\in (V^\vee \setminus \{0\})\times \Omega_0(G,H)\times \R$ satisfying $\pairing{\xi}{gv_0}=\log \chi(g)+c$ for any $g\in G$. 
\end{enumerate}
Here, $f_\xi(g):=\pairing{\xi}{gv_0-v_0}$ for $g\in G$. 
\end{theorem}
We prove this theorem in Section~\ref{sec:theorem_main}.

Moreover, we also give necessary conditions for the injectivity of (\ref{eq:correspondence}). 
To state them, we prepare the notion of cyclic. 
\begin{definition}[cyclic]
We say a vector $v\in V$ is cyclic if 
$\spanvec \{ gv\, |\, g\in G\}=V$.  
\end{definition}

\begin{proposition}\label{prop:proper_affine_subspace}
If the correspondence (\ref{eq:correspondence}) is injective, 
then the following equivalent conditions (A) and (B) are satisfied. 
Namely, ((\ref{eq:correspondence}) is injective)$\Rightarrow$ (A) $\Leftrightarrow$ (B). 
\begin{enumerate}
\item[(A)] The orbit $G v_0$ is not contained in any proper affine subspace of $V$. 
\item[(B)] 
\begin{enumerate}
\item[(1)] $v_0\in V$ is cyclic, 
\item[(2)] $\rho^\vee:G\to GL(V^\vee)$ has no nonzero $G$-fixed vector.  
\end{enumerate}
\end{enumerate}
Here $\rho^\vee$ is the contragredient representation of $G$. 
Moreover, in the case where $\Omega_0(G,H)=\{1\}$, 
the converse implication also holds. 
%one of the conditions (A) and (B) above (that is, both of them) imply the injectivity of the correspondence (\ref{eq:correspondence}).  
\end{proposition}
We prove this proposition in Section~\ref{subsec:proof_of_prop}
\begin{comment}
\begin{lemma}\label{lem:nec_cond_rep}
%Let $\rho:G\to GL(V)$ be a finite dimensional real representation of a Lie group $G$. 
In the same setting as in Section~\ref{sec:setting}. 
The conditions (A) is equivalent to the following condition (B):
%\begin{enumerate}
%\item[(A)] $Gv_0$ is not contained in any proper affine subspace of $V$. 
%\end{enumerate}
\end{lemma}

\begin{proof}
\end{proof}
\end{comment}

\begin{comment}
\begin{corollary}\label{cor:cyc_contra}
If the correspondence (\ref{eq:correspondence}) is injective, 
then the following conditions hold. 
\begin{enumerate}
\item[\rm (1)] $v_0\in V$ is cyclic. 
\item[\rm (2)] The contragredient representation $\rho^\vee:G\to GL(V^\vee)$ has no nonzero $G$-fixed vectors. 
\end{enumerate}
\end{corollary}
\end{comment}

\begin{remark}\label{rem:omega_trivial}
In the case where $G$ is compact or connected semisimple, 
we have $\Omega_0(G,H)=\{1\}$. 
See \cite{TY18} for the details. 
\end{remark}

\begin{comment}
\begin{proof}[Corollary~\ref{cor:cyc_contra}]
First, we prove the former part. 
\begin{enumerate}
\item[\rm (1)] Assume $v_0$ is not cyclic. 
Then there exists a triple $(\xi, 1, 0)\in (V^\vee\setminus\{0\})\times \Omega_0(G,H)\times \R$ satisfying $\pairing{\xi}{gv_0}=0=\log 1+0$. 
\item[\rm (2)] Assume there exists a $G$-fixed vector $\xi\in V^\vee\setminus\{0\}$. 
Then the triple $(\xi, 1, \pairing{\xi}{v_0})\in (V^\vee\setminus \{0\})\times \Omega_0(G,H)\times \R$ satisfies $\pairing{\xi}{gv_0}=\pairing{\rho^\vee(g)\xi}{v_0}=\log 1+ \pairing{\xi}{v_0}$. 
\end{enumerate}
\end{proof}
\end{comment}

\subsection{Equivalence}
We use the same notation %work in the setting 
as in Section~\ref{sec:setting}. 
In this subsection, we give an answer to Question~\ref{question:equiv_rel}. 
To state it, we introduce the notations 
$\tilde{\mathcal{V}}(G)$ and $\tilde{\mathcal{V}}(G,H)$. 
\begin{definition}
We put 
\begin{align*}
\tilde{\mathcal{V}}(G)&:=\{ (V,v_0)\ |\ V \text{ is a finite dimensional real representation of }G, v_0\in V \text{ is cyclic}\}, \\
\tilde{\mathcal{V}}(G,H)&:=\{ (V,v_0)\in \tilde{\mathcal{V}}(G)\ |\  v_0\in V^H \}. 
\end{align*}

We say elements $(V,v_0)$ and $(V',v_0')$ in $\tilde{\mathcal{V}}(G)$ are equivalent if there exists a $G$-equivariant linear isomorphism $\psi:V\to V'$ such that $\psi(v_0)=v_0'$ and denote it by $(V,v_0)\sim (V',v_0')$. 
This is an equivalence relation on $\tilde{\mathcal{V}}(G)$. 
By definition, this is also an equivalence relation on $\tilde{\mathcal{V}}(G,H)$. 
%Then we denote the quotient sets $\tilde{\mathcal{V}}(G)/\sim$ and $\tilde{\mathcal{V}}(G,H)/\sim$ by $\mathcal{V}(G)$ and $\mathcal{V}(G,H)$, respectively. 
\end{definition}

\begin{theorem}\label{theorem:equiv}
Equivalent elements in $\tilde{\mathcal{V}}(G,H)$ generate the same family by our method. 
\end{theorem}
We prove this theorem in Section~\ref{sec:theorem_equiv}. 

\begin{remark}
From Theorem~\ref{theorem:equiv}, 
in the special case $\dim V^H=1$, 
the choice of $v_0$ is essentially unique. 
In the next section, %Section~\ref{sec:example}, 
we also see an example in which the choice of $v_0$ is essentially unique even if $\dim V^H>1$. 
\end{remark}

\begin{comment}
\begin{remark}
The converse implication of Theorem~\ref{theorem:equiv} does not hold even when $\Omega_0(G,H)=\{1\}$. 
Namely, elements in $\tilde{V}(G,H)$ that are not equivalent can generate the same family. 
In fact, for $G/H=G=\{\begin{pmatrix}e^{i\theta} & \beta \\ 0 &1\end{pmatrix}\ |\ \theta\in \R, \beta\in \C\}$, 
though the following $(V,v_0)$ and $(V,v_0')$ generate the same family, 
they are not equivalent. 
\begin{align}
V=\Sym(2,\C)\\
\rho:G\to GL(V), \rho(g)S=\trans g^{-1}Sg^{-1}\quad (S\in V),\\
v_0:=\begin{pmatrix}1 & 0\\0 & 0\end{pmatrix}, v_0':=\begin{pmatrix} 1 &0 \\0&1\end{pmatrix}
\end{align}
\end{remark}
%‰ÂÏ•ª—̈悪‹óW‡'©'àc
\end{comment}

\section{Generalized inverse Gaussian distribution}\label{sec:example}
Throughout this section, we put $G=\R_{>0}$, $H=\{1\}$ and $V=\R^2$, and consider a representation $\rho\colon G\to GL(V)$ given by
$\rho(g)=\begin{pmatrix} g & \\ & g^{-1}\end{pmatrix}$ for $g\in G$.
We answer Questions~\ref{question:injective} and \ref{question:equiv_rel} for this case. 

We consider the following two cases. \\
%\begin{enumerate}
%\item 
(Case~1) In the case where $\begin{pmatrix}r\\s \end{pmatrix}\in V^H=V$ with $r=0$ or $s=0$: \\
Vectors $\begin{pmatrix}r\\0\end{pmatrix}$, $\begin{pmatrix}0\\s\end{pmatrix}$ are not cyclic. Therefore the obtained families have ``unessential parameters''. \\
%\item 
(Case~2) In the case where $\begin{pmatrix}r\\s \end{pmatrix}\in V^H$ with $r\neq0$ and $s\neq 0$: 
\begin{proposition}\label{prop:GIG}
The pairs $(V,\begin{pmatrix}r\\s\end{pmatrix})$ with $r\neq 0$ and $s\neq 0$ are equivalent each other. 
Moreover, we obtain the family $\{dp_{a,b,\lambda}\}_{(a,b,\lambda)\in\Theta}$ of GIG (\ref{def:GIG}) by applying our method to $(V,\begin{pmatrix}r\\s\end{pmatrix})$, 
where $\Theta=\{(a,b,\lambda)\in \R^3\ |\ (a,b,\lambda) \text{ satisfies }(\ref{eq:param_GIG})\}$.
\end{proposition}

\begin{definition}[Generalized inverse Gaussian distribution. See \cite{j82} for the details]
The following distribution on $\R_{>0}$ is called generalized inverse Gaussian distribution. 
\begin{align}
%\frac{(a/b)^\frac{\lambda}{2}}{2K_\lambda(\sqrt{ab})}
c_{a,b,\lambda}x^{\lambda-1}e^{-(ax+b/x)/2}dx\quad (x\in \R_{>0}),\label{def:GIG}
\end{align}
where $dx$ denotes Lebesgue measure on $\R_{>0}$,  and $(a,b,\lambda)$ satisfies one of the following three conditions:
\begin{align}
(i)\ a>0, b>0,\  (ii)\ a>0, b=0, \lambda>0,\ (iii)\ a=0, b>0, \lambda<0. \label{eq:param_GIG}
%(a>0, b\geq 0, \lambda >0)\text{ or }(a>0, b>0, \lambda=0)\text{ or }(a\geq 0, b>0, \lambda<0). 
\end{align}
Here $c_{a,b,\lambda}$ is the normalizing constant given as follows, respectively. 
\begin{align}
(i)\ \frac{(a/b)^\frac{\lambda}{2}}{2K_\lambda(\sqrt{ab})},\ (ii)\ \frac{1}{\Gamma(\lambda)}\left(\frac{a}{2}\right)^\lambda,\ (iii)\ \frac{1}{\Gamma(-\lambda)}\left(\frac{b}{2}\right)^{-\lambda}, 
\end{align}
where $K_\lambda$ is the modified Bessel function of the second kind with index $\lambda$. 
\end{definition}
\begin{proof}[Proposition~\ref{prop:GIG}]
Put $v_0:=\frac{1}{2}\begin{pmatrix}1\\1\end{pmatrix}$. 
For $r,s\neq 0$, a $G$-linear isomorphism $\begin{pmatrix} 2r & 0\\0 & 2s\end{pmatrix}\in GL(V)$ gives $(V,v_0)\sim (V,\begin{pmatrix}r\\s\end{pmatrix})$, which implies the former part. 
%The former part follows $(V,v_0)\sim (V,\begin{pmatrix}x\\y\end{pmatrix})$ for $x,y\neq 0$ by taking a $G$-linear isomorphism $\begin{pmatrix} 2x & 0\\0 & 2y\end{pmatrix}\in GL(V)$. 
%For the former part, it is enough to show $(V,v_0)\sim (V,\begin{pmatrix}x\\y\end{pmatrix})$ for $x,y\neq 0$. 
%This follows by taking a $G$-linear isomorphism . 

For the latter part, it is enough to show the case $(V,v_0)$ by Theorem~\ref{theorem:equiv}.  
It is easily checked that $\Omega_0(G,H)=\{ x\mapsto x^\lambda\ |\ \lambda\in \R\}$. 
Take a relatively invariant measure $\frac{dx}{x}$ on $\R_{>0}$. 
We identify $(\R^2)^\vee$ with $\R^2$ by taking the standard inner product. 
Then we have 
\begin{align*} d\tilde{p}_{a,b,\lambda}(x)&:= \exp(-\pairing{\begin{pmatrix}a\\b\end{pmatrix}}{\begin{pmatrix}x&\\& x^{-1}\end{pmatrix}v_0})x^\lambda\frac{dx}{x} \quad (\begin{pmatrix}a\\ b\end{pmatrix}\in \R^2)\\
&=\exp(-(ax+bx^{-1})/2)x^{\lambda-1}dx. 
\end{align*}
We get $\Theta=\{\theta=(a,b,\lambda)\in \R^3\ |\ (a,b,\lambda)\text{ satisfies }(\ref{eq:param_GIG})\}$. 
By normalizing these distributions, 
%$\varphi(\theta)=\log  \frac{2K_\lambda(\sqrt{ab})}{(a/b)^\frac{\lambda}{2}}$. 
we obtain the desired family of GIG (\ref{def:GIG}). 
\end{proof}

Finally, let us check the injectivity of the correspondence (\ref{eq:correspondence}). 
For $(a,b,c,\lambda)\in~\R^4$, 
\begin{align*}
axg+byg^{-1}=\lambda \log g+c \quad \text{for any }g\in G
\end{align*}
holds only if $(a,b,c,\lambda)=0$. Thus, the condition (iii) of Theorem~\ref{theorem:main} is satisfied. 

\section{Proof of main theorems}
In this section, we give proofs to Theorems~\ref{theorem:main} and \ref{theorem:equiv} and Proposition~\ref{prop:proper_affine_subspace}. 
\subsection{Preliminary}
In this subsection, we prepare some notations for proofs in the following sections. 
Let $G$ be a Lie group, H a closed subgroup of $G$ and V a finite dimensional real vector space. 
\begin{notation}
We denote by $C(G)$ the vector space consisting of all $\R$-valued continuous functions on $G$. 
The constant function $1$ is an element of $C(G)$. 
The space $C(G)$ admits 
the left and right regular representations $L$, $R: G\to GL(C(G))$, respectively. 
We put $C(G)^H:=\{f\in C(G)\ |\ R_hf=f \text{ for any }h\in H\}$. 
%: for $g\in G$, 
%\begin{align}
%(L_gf)(x)&=(L(g)f)(x):=f(g^{-1}x), \\
%(R_gf)(x)&=(R(g)f)(x):=f(xg) \quad (f\in C(G), x\in X). 
%\end{align}
%We consider the following subspace of $C(G)$: 
%\begin{align}
%\log \Omega_0(G,H):=\{\log \chi:G\to \R\ |\ \chi \in \Omega_0(G,H) \}\subset C(G). 
%\end{align}
%For $v_0\in V^H$ and $\xi\in V^\vee$, we define $f_\xi\in C(G)^H$ as follows:
%\begin{align}f_\xi(g):=\pairing{\xi}{gv_0-v_0}. \end{align}
\end{notation}

\begin{remark}\label{rem:in_log_Omega}
The set $\log\Omega_0(G,H)$ is a subspace of $C(G)$ (see (\ref{eq:log_omega})). 
For $f\in C(G)$, the condition $f\in \log\Omega_0(G,H)$ is equivalent to the pair of the following conditions:
\begin{enumerate}
\item[$(a)$] \ $f(h)=0$ for any $h\in H$, 
\item[$(b)$] \ $f(gg')=f(g)+f(g')$ for any $g,g'\in G$. 
\end{enumerate}
%Moreover, for $\xi\in V^\vee$ and $f=f_\xi\in C(G)^H$, the condition $(b)$ above is equivalent to the following condition:
%\begin{align}
%\pairing{\xi}{gg'v_0-gv_0-g'v_0+v_0}=0 \text{ for any }g,g'\in G. 
%\end{align}
\end{remark}

\begin{notation}
%\begin{enumerate}
%\item 
We denote by $\ev$ the evaluation map. 
We identify $V$ with $(V^\vee)^\vee$ canonically as follows:
\begin{align}
V\to (V^\vee)^\vee,\ x\mapsto \ev_x. 
\end{align}
%\item 
Let $W$ be a subspace of $V$. Then we put
\begin{align}
W^\perp:=\{f\in V^\vee\ |\ \pairing{f}{w}=0 \text{ for any }w\in W \}. 
\end{align}
%\end{enumerate}
%Let $S$ be a set and $\mathcal{F}(S)$ a family of $\R$-valued functions on $S$. 
%For $s\in S$, we consider the evaluation map $\ev_s$ as follows: 
%\begin{align}
%\ev_s: \mathcal{F}(S)\to \R, f\mapsto f(s). 
%\end{align}
%In the case where $S:=V$ is a finite dimensional vector space and $\mathcal{F}(S):=V^\vee$, 
\end{notation}

\begin{notation}
For a representation $\rho:G\to GL(V)$, we denote the contragredient representation by $\rho^\vee:G\to GL(V^\vee)$. 
We often use simpler notation $g^\vee \xi:=\rho^\vee(g)\xi$ for $g\in G$ and $\xi\in V^\vee$. 
Then, the following equality holds:
\begin{align}
\pairing{g^\vee \xi}{v}=\pairing{\xi}{g^{-1}v} \quad (g\in G,\ v\in V,\ \xi\in V^\vee). \label{eq:contragredient}
\end{align}
\end{notation}

\subsection{Proof of Theorem~\ref{theorem:main}}\label{sec:theorem_main}
\begin{proof}[Theorem~\ref{theorem:main}]
%We consider the following condition (iv), which is equivalent to $\lnot$(ii) by Remark~\ref{rem:in_log_Omega}. 
%\[\text{(iv) There exists }\xi\in V^\vee\setminus\{0\} \text{ such that } f_\xi\in \log\Omega_0(G,H). \]
We are enough to show $\lnot$(ii)$\Rightarrow \lnot$(iii)$\Rightarrow \lnot$(i)$\Rightarrow\lnot$(ii). 

First, we see $\lnot$(ii)$\Rightarrow \lnot$(iii). 
Take $\xi\in V^\vee\setminus\{0\}$ such that $f_\xi\in \log\Omega_0(G,H)$. 
Then there exists $\chi\in \Omega_0(G,H)$ such that 
$\pairing{\xi}{gv_0-v_0}=\pairing{\xi}{gv_0}-\pairing{\xi}{v_0}=\log \chi(g)$ for any $g\in G$, 
so $\lnot$(iii) is proved. 

Next, we see $\lnot$(iii)$\Rightarrow \lnot$(i). 
Assume there exist $\xi\in V^\vee\setminus\{0\}$, $c\in \R$ and $\chi\in \Omega_0(G,H)$ satisfying $\pairing{\xi}{gv_0}=\log \chi(g)+c$ for any $g\in G$. 
Take any $\theta_1=(\xi_1, \chi_1)\in \Theta$ and put $\theta_2:=(\xi_1+\xi,\chi_1\chi)\in V^\vee\times \Omega_0(G,H)$. 
It is enough to show that $\theta_2\in \Theta$ and $p_{\theta_1}=p_{\theta_2}$. 
This comes from 
$d\tilde{p}_{\theta_2}(x)=e^{-\pairing{\xi_1+\xi}{xv_0}}\chi_1(x)\chi(x)d\mu(x)=e^{-\pairing{\xi}{xv_0}+\log \chi(x)}e^{-\pairing{\xi_1}{xv_0}}\chi_1(x)d\mu(x)=e^{-c}d\tilde{p}_{\theta_1}(x)$. 

\begin{comment}
First we show $\lnot (ii) \Rightarrow \lnot (i)$. 
Assume $\spanvec_\R\{gg'v_0-gv_0-g'v_0+v_0 | g\in G\}\subsetneq V$. 
Then we can and do take nonzero vector $\xi\in V^\vee$ such that 
\begin{align} 
\pairing{\xi}{gg'v_0-gv_0-g'v_0+v_0}=0\quad\text{ for any }g,g'\in G. \label{eq:assump_pair_zero}
\end{align}
Take any $\theta_0=(\xi_0,\chi_0)\in \Theta$ and put 
$\theta_1=(\xi_1,\chi_1):=(\xi_0+\xi, e^f \chi_0)$, where $f:G\to \R$, $g\mapsto \pairing{\xi}{gv_0-v_0}$. 
\begin{claim}
\begin{enumerate}
\item[(a)] $f\in \log \Omega_0(G,H):=\{ \log\chi:G\to \R\ |\ \chi\in \Omega_0(G,H)\}$. 
\item[(b)] $d\tilde{p}_{\theta_1}=e^{-\pairing{\xi}{v_0}}d\tilde{p}_{\theta_0}$. 
\end{enumerate}
\end{claim}
Claim (b) implies $\theta_1\in \Theta$ and $p_{\theta_1}=p_{\theta_0}$, 
so the correspondence (\ref{eq:correspondence}) is not injective. 
Thus, it is enough to show Claim above. 
\begin{proof}[Proof of Claim]
For the first condition (a), it is clear that $f(h)=0$ for any $h\in H$, so we show the condition that $f(gg')=f(g)+f(g')$ for any $g,g'\in G$. 
By (\ref{eq:assump_pair_zero}), we have 
\begin{align*}
f(gg')&=\pairing{\xi}{gg'v_0-v_0}\\
&=\pairing{\xi}{gv_0+g'v_0-2v_0}\\
&=\pairing{\xi}{gv_0-v_0}+\pairing{\xi}{gv_0'-v_0}\\
&=f(g)+f(g'). 
\end{align*}
The second condition (b) can be checked by direct calculation, so we omit the proof. 
\end{proof}
\end{comment}
Finally, we see $\lnot$(i)$\Rightarrow\lnot$(ii). 
Assume two distinct elements $\theta_1=(\xi_1,\chi_1)$ and $\theta_2=(\xi_2,\chi_2)\in \Theta$ satisfy $p_{\theta_1}=p_{\theta_2}$. 
Put $\xi:=\xi_2-\xi_1$. 
It is enough to show the following:
\begin{claim}
$\xi\neq 0$ and $f_\xi\in \log\Omega_0(G,H)$. 
\end{claim}
From $p_{\theta_1}=p_{\theta_2}$, we have for almost every $x\in X$, 
\begin{align*}
\exp(-\pairing{\xi_1}{x v_0}+\log \chi_1(x)-\varphi(\theta_1)+\pairing{\xi_2}{xv_0}-\log \chi_2(x)+\varphi(\theta_2))=\frac{dp_{\theta_1}}{dp_{\theta_2}}(x)=1. 
\end{align*}
Therefore we have 
\begin{align}
\pairing{\xi}{gv_0}+\varphi(\theta_2)-\varphi(\theta_1)=\log\chi_2(g)-\log\chi_1(g)\in \log \Omega_0(G,H). \label{eq:injective}
\end{align}
From Remark~\ref{rem:in_log_Omega}$(a)$, we have $\varphi(\theta_2)-\varphi(\theta_1)=-\pairing{\xi}{v_0}$, that is, $f_\xi\in \log\Omega_0(G,H)$. 
Moreover, from (\ref{eq:injective}) and $\theta_1\neq \theta_2$, we obtain $\xi\neq 0$. 
\end{proof}

\subsection{Proof of Proposition~\ref{prop:proper_affine_subspace}}\label{subsec:proof_of_prop}
In this subsection, we prove Proposition~\ref{prop:proper_affine_subspace} by using Lemma~\ref{lem:property_V_xi} below. 
\begin{lemma}\label{lem:property_V_xi}
For $\xi\in V^\vee\setminus \{0\}$, %we put 
%\[ V_\xi:= \{v\in V\ |\ \pairing{\xi}{v-v_0}=0\}. \]
we consider the following three conditions:
\begin{enumerate}
\item[(i)] $g^\vee\xi = \xi$ for any $g\in G$, 
\item[(ii)] $f_\xi=0$ (see Theorem~\ref{theorem:main} for the definition of $f_\xi$), 
\item[(iii)] there exists $c\in \R$ satisfying $Gv_0\subset \{ v\in V\ |\ \pairing{\xi}{v}=c\}$. 
\end{enumerate}
Then, we have (i)$\Rightarrow$(ii)$\Leftrightarrow$(iii). 
Moreover, under the assumption that $v_0$ is cyclic, the implication (iii)$\Rightarrow $(i) also holds. 
\end{lemma}
\begin{proof}
Since the implications (i)$\Rightarrow$(ii)$\Leftrightarrow$(iii) are easy, 
we prove only the implication (iii)$\Rightarrow$(i) under the assumption that $v_0$ is cyclic. 
Take any $g\in G$. 
It is enough to show that $\pairing{g^\vee \xi}{g'v_0}=\pairing{\xi}{g'v_0}$ for any $g'\in G$. 
From (\ref{eq:contragredient}), we have
\[ \pairing{g^\vee \xi}{g'v_0}=\pairing{\xi}{g^{-1}g'v_0}=c=\pairing{\xi}{g'v_0}. \]
\end{proof}

\begin{proof}[Proposition~\ref{prop:proper_affine_subspace}]
First, note that we have the following three easy implications (a), (b) and (c):
\begin{enumerate}
\item[(a)] \ $\lnot$(A)$\iff$ there exists $\xi\in V^\vee\setminus\{0\}$ satisfying Lemma~\ref{lem:property_V_xi}(iii), 
\item[(b)] \ $\lnot$(B)(2)$\iff$ there exists $\xi\in V^\vee\setminus\{0\}$ satisfying Lemma~\ref{lem:property_V_xi}(i), 
\item[(c)] \ (A)$\implies$ $v_0$ is cyclic. 
\end{enumerate}
Therefore, the equivalence (A)$\Leftrightarrow$(B) comes from Lemma~\ref{lem:property_V_xi}. 

Next, the implication ((\ref{eq:correspondence}) is injective)$\Rightarrow$(A) follows from (a). 
In fact, the condition Theorem~\ref{theorem:main}(ii) fails if 
there exists $\xi\in V^\vee\setminus\{0\}$ satisfying Lemma~\ref{lem:property_V_xi}(ii). 
 
Finally, assume $\Omega_0(G,H)=\{1\}$. 
The converse implication above also holds. 
So, (A) implies the injectivity of (\ref{eq:correspondence}). 
\end{proof}

\subsection{Proof of Theorem~\ref{theorem:equiv}}\label{sec:theorem_equiv}
%In this subsection, we prove Theorem~\ref{theorem:equiv}. 
We show Theorem~\ref{theorem:equiv} by using Lemmas~\ref{lemma:one_to_one_corresp} and \ref{lemma:H-correspondence} below. 
We prove Lemma~\ref{lemma:one_to_one_corresp} in the next subsection. %after proof of Theorem~\ref{theorem:equiv}. 
\begin{proof}[Theorem~\ref{theorem:equiv}]
It is enough to show that $\{g\mapsto \pairing{\xi}{gv_0}\ |\ \xi \in V^\vee\}=\{g\mapsto \pairing{\xi'}{gv_0'}\ |\ \xi' \in V'^\vee\}$ as a subspace of $C(G)^H$ 
if $(V,v_0), (V',v_0')\in \tilde{\mathcal{V}}(G,H)$ are equivalent.  
This follows from Lemmas~\ref{lemma:one_to_one_corresp} and \ref{lemma:H-correspondence} below. 
\end{proof}

\begin{lemma}\label{lemma:one_to_one_corresp}
Put 
\begin{align*} 
\mathcal{V}(G)&:=\tilde{\mathcal{V}}(G)/\sim, \\
\mathcal{W}(G)&:=\{ W \subset C(G)\ |\ W \text{ is a finite dimensional }L_G\text{-invariant subspace}\}. \end{align*}
The following map gives a one-to-one correspondence. 
\begin{align}
\mathcal{V}(G)\to \mathcal{W}(G),\ (V,v_0)\mapsto \eta(V^\vee),
\end{align}
where
\begin{align}
\eta:=\eta_{V,v_0}:V^\vee \to C(G),\ \xi \mapsto (g\mapsto \pairing{\xi}{gv_0}). \label{eq:eta}
\end{align}
\end{lemma}
\begin{lemma}\label{lemma:H-correspondence}
Let $H$ be a closed subgroup of $G$. 
Suppose $(V,v_0)\in \mathcal{V}(G)$ corresponds to $W\in \mathcal{W}(G)$ in Lemma~\ref{lemma:one_to_one_corresp}. 
Then $v_0$ is $H$-fixed if and only if 
any element $w\in W$ is $R_H$-fixed. 
\end{lemma}
\begin{proof}
We have
\begin{align*}
&\text{the function } \eta(\xi):G\to \R \text{ is }R_H\text{-fixed for any }\xi \in V^\vee, \\
\iff & \pairing{\xi}{ghv_0}=\pairing{\xi}{gv_0}\ \text{ for any }g\in G, h\in H\text{ and } \xi\in V^\vee, \\
\iff & ghv_0=gv_0 \text{ for any } g\in G\text{ and } h\in H, \\
\iff & v_0\text{ is }H\text{-fixed}. 
\end{align*}
\end{proof}

\subsection{Proof of Lemma~\ref{lemma:one_to_one_corresp}}
In this subsection, we prove Lemma~\ref{lemma:one_to_one_corresp}. 
To show this lemma, we use Lemmas~\ref{lem:property_of_eta} and \ref{lem:ev_cyclic} below. 
%We show the lemmas after the proofs of Lemmas~\ref{lemma:one_to_one_corresp} and \ref{lemma:H-correspondence}. 

\begin{lemma}[property of $\eta$]\label{lem:property_of_eta}
The map $\eta:V^\vee\to C(G)$ defined in $(\ref{eq:eta})$ satisfies the following:
\begin{enumerate}
\item[$(1)$] $\eta$ is a $G$-equivariant linear map, 
\item[$(2)$] $v_0$ is cyclic if and only if $\eta$ is injective, 
%\item[(3)] For a subgroup $H\subset G$, 
%$v_0$ is $H$-fixed if and only if $\eta(\theta)$ is $R_H$-invariant for any $\theta\in V^\vee$. 
\item[$(3)$] $(V,v_0)\sim (V',v_0')\Rightarrow \eta(V^\vee)=\eta'(V'^\vee)$, where $\eta=\eta_{V,v_0}$ and $\eta'=\eta_{V',v_0'}$. 
\end{enumerate}
\end{lemma}
We give a proof of this lemma at the end of this subsection. 
\begin{lemma}\label{lem:ev_cyclic}
Let $W\subset C(G)$ be a finite dimensional $L_G$-invariant subspace. 
Then $v_0:=\ev_e|_W\in W^\vee$ is $L_G^\vee$-cyclic in $W^\vee$. 
\end{lemma}
\begin{proof}%[Lemma~\ref{lem:ev_cyclic}]
Put $E:=\spanvec\{ L_g^\vee v_0\ |\ g\in G\}\subset W^\vee$. 
It is enough to show $E^\perp=\{0\}$. 
Take any function $f\in E^\perp$, 
then we have 
$f(g)=(L_{g^{-1}}f)(e)=\pairing{v_0}{L_{g^{-1}}f}=\pairing{L_g^\vee v_0}{f}=0$. 
Therefore, we obtain $f=0$. 
\end{proof}

%\begin{proof}[Proof of Lemma~\ref{lemma:H-correspondence}]
%Lemma~\ref{lemma:H-correspondence} follows from Lemma~\ref{lem:property_of_eta} (3) immediately. 
%\end{proof}

\begin{proof}[Lemma~\ref{lemma:one_to_one_corresp}]
From Lemmas~\ref{lem:property_of_eta}(1) and \ref{lem:ev_cyclic}, 
the following maps are well-defined:
\begin{align}
\Phi: \tilde{\mathcal{V}}(G)&\to \mathcal{W}(G), &(V,v_0)&\mapsto \eta(V^\vee),\\
\Psi: \mathcal{W}(G)&\to \tilde{\mathcal{V}}(G), &W&\mapsto (W^\vee, \ev_e|_W). 
\end{align}
Then it is enough to show the following:
\begin{enumerate}
\item[(a)] $(V,v_0)\sim (V',v_0')$ in $\tilde{\mathcal{V}}(G)\Rightarrow \Phi(V,v_0)=\Phi(V',v_0')$, 
\item[(b)] $\Phi\circ \Psi=\id_{\mathcal{W}(G)}$, 
\item[(c)] $\Psi\circ \Phi (V,v_0)\sim (V,v_0)$ in $\tilde{\mathcal{V}}(G)$ for $(V,v_0)\in \tilde{\mathcal{V}}(G)$. 
\end{enumerate}

First, the condition (a) follows from Lemma~\ref{lem:property_of_eta}(3). 

Next, we show the condition (b). 
Let $W$ be an element of $\mathcal{W}(G)$. 
Since we have $\Psi(W)=(W^\vee, \ev_e|_W)$, 
we get $\Phi\circ \Psi(W)=\{ g\mapsto \pairing{\xi}{L_g^\vee (\ev_e|_W)}\ |\ \xi \in (W^\vee)^\vee\}$. 
Then, we have
\begin{align*}
\pairing{\xi}{L_g^\vee (\ev_e|_W)}
=(L_g^\vee (\ev_e|_W))(\xi)
=(\ev_e|_W)(L_{g^{-1}}\xi)
=(L_{g^{-1}}\xi)(e)
= \xi(g). 
\end{align*}
Therefore, we obtain $\Phi\circ \Psi(W)=W$. 

Finally, we show the condition (c). 
Let $(V,v_0)$ be an element of $\tilde{\mathcal{V}}(G)$. 
Put $W:=\eta(V^\vee)$ and $(V',v_0'):=\Psi \circ \Phi (V,v_0)=\Psi(W)=(W^\vee,\ev_e|_{W})$. 
Since $\eta^\vee:W^\vee\to (V^\vee)^\vee$ is a $G$-linear isomorphism by Lemma~\ref{lem:property_of_eta}(1) and (2), 
it is enough to show that 
$\eta^\vee(\ev_e|_{W})=v_0$. 
%by identifying $V$ with $(V^\vee)^\vee$ by $v\mapsto \ev_v$. 
For any $\xi\in V^\vee$, we have 
\begin{align}
\pairing{\xi}{\eta^\vee(\ev_e|_{W})}
=\pairing{\eta(\xi)}{\ev_e|_{W}}
=\eta(\xi)(e)=\pairing{\xi}{v_0}. 
\end{align}
Therefore, we obtain $\eta^\vee(\ev_e|_{W})=v_0$. 
\end{proof}

\begin{proof}[Lemma~\ref{lem:property_of_eta}]
\begin{enumerate}
\item[(1)] Clearly, $\eta$ is a linear map. 
The $G$-equivariance of $\eta$ follows from the definition of the contragredient representation. 
\item[(2)] Since $\eta$ is linear, it is enough to show that $v_0$ is cyclic if and only if $\ker \eta=\{0\}$. 
The condition $\ker \eta=\{0\}$ means that for $\xi\in V^\vee$, $\pairing{\xi}{gv_0}=0$ for any $g\in G$ implies $\xi=0$. 
Therefore this is equivalent to the condition $v_0$ is cyclic. 
\item[(3)] Take a $G$-equivariant linear isomorphism $\psi:V\to V'$ with $\psi(v_0)=v_0'$. 
Then it is enough to show $\eta'=\eta \circ \psi^\vee:V'^\vee\to C(G)$. 
For any $\xi'\in V'^\vee$ and $g\in G$,  
\begin{align*}
\eta\circ \psi^\vee(\xi')(g)
=\pairing{\psi^\vee \xi'}{gv_0}
=\pairing{\xi'}{\psi(gv_0)}
=\pairing{\xi'}{g\psi(v_0)}
=\pairing{\xi'}{gv_0'}=\eta'(\xi')(g). 
\end{align*}
\end{enumerate}
\end{proof}
\section*{Acknowledgements}
The authors would like to thank Dr. Fr\'ed\'eric Barbaresco for recommending us to submit a paper to the conference Geometric Science of Information 2019. 
The authors wish to thank referees for several helpful comments, 
particularly the comment concerning the condition (A) in Proposition~\ref{prop:proper_affine_subspace}.

\end{document}